\documentclass[a4paper,10pt]{article}

\usepackage{amsmath,amssymb,amsthm, bm,graphicx}
\usepackage[all]{xy}
\usepackage[symbol]{footmisc}
\usepackage{mymath}
\makeatletter
    
    \@addtoreset{equation}{section}
  \makeatother%

\begin{document}
\title{On the Brauer group of affine diagonal quadrics}
\author{\textsc{Tetsuya Uematsu}}
\date{}
\maketitle
\begin{abstract}
In a previous work, we introduced the notion of uniform generators of the
 Brauer group and proved that general diagonal cubic surfaces do not
 have such generators. 

In this paper, we prove that a similar non-existence result holds for affine diagonal quadrics.
\end{abstract}
\section{Introduction}\label{S-i}
\footnote[0]{Date: \today.}
\footnote[0]{%
2010 Mathematics Subject Classification. 
Primary: 14F22. 
Secondary: 11R34, 14J26, 19F15. 
}
Let $X$ be a variety over a field $k$ of characteristic zero. Explicit description of elements
of the Brauer group $\Br(X)$ of $X$ has been studied by many
authors. Historically, the Brauer group of $k$ was first considered and it is
described for example by norm residue symbols and cyclic algebras. In 1970's,
Manin~\cite{manin1986cubic} considered this problem for diagonal cubic surfaces $X$,
that is, projective surfaces defined by a homogeneous equation of the
form $ax^3+by^3+cz^3+dt^3=0$. Under the assumption that $k$ contains a
primitive cubic root $\zeta$ of unity, let $\{\cdot,\cdot\}_3$ be the
follwoing norm residue symbol 
\[
 \{\cdot,\cdot\}_3\colon K_2^M(k(X)) \to H^2(k(X),\mu_3^{\otimes 2}) \cong
 H^3(k(X),\mu_3) \inj \Br(k(X)),
\] 
where $k(X)$ is the function field of $X$. For surfaces $X$ of the form $x^3+y^3+z^3+dt^3=0$, he proved that the following two elements 
\[
 e_1=\left\{d,\dfrac{x+\zeta y}{x+y}\right\}_3, \quad e_2=\left\{d, \dfrac{x+z}{x+y}\right\}_3
\]
are naturally considered as elements in $\Br(X)$ and generates
$\Br(X)/\Br(k) \cong (\bZ/3\bZ)^2$. Some generalizations of this result can
be found in for example, \cite{saito2009zerocycle}, Proposition 4.2.6
and \cite{uematsu14:_brauer}.

We do stress that the above generators are algebraically ``uniform'' in
the following sense. If we put 
\begin{equation*}
e_1(D)=\left\{D,
 \dfrac{x+\zeta y}{x+y}\right\}_3, \quad e_2(D)=\left\{D,\dfrac{x+z}{x+y}\right\}_3
\end{equation*}
where $D$ is considered as an {\it indeterminate} and if we want
 symbolic generators of $\Br(X_d)/\Br(k)$, where $X_d$ is the surface
 defined by $x^3+y^3+z^3+dt^3=0$ with $d \in k^{\ast}$, we can get them
 by specializing $e_1(D)$ and $e_2(D)$ at $D=d$. The similar result
 holds for the family $X_{c,d}: x^3+y^3+cz^3+dt^3=0$ with
 $\Br(X_{c,d})/\Br(k) \cong \bZ/3\bZ$~\cite{uematsu14:_brauer}.  In
 general, it is not necessarily that for a given family of varieties,
 generators of the Brauer group of each variety in the family is given
 by specializing ``uniform'' elements algebraically expressed by
 parameters of this family. In fact, for the Brauer group of general
 diagonal cubic surfaces $x^3+by^3+cz^3+dt^3=0 $ parametrized by three
 coefficients $b,c,d$, we have no such uniform generators. For details,
 see \cite{uematsu14:_brauer}. We remark that any cubic surface in the
 first two families has rational points, though the existence of
 rational points of general diagonal cubic surfaces heavily depends on
 coefficients $b,c,d$. Based on these observations, it seems that the
 non-existence of such uniform generators relates the complexity of a given
 family in some extent.

For this problem, Timothy D.~Browning asked the author whether such a uniform
generator does not exist in the case of {\it affine diagonal
quadrics}. The main objective of this article is to answer his question: 
as is the case of diagonal cubic surfaces, we have no uniform generator
in this case, too.  To state the claim more precisely, we prepare some
notations. Let $F=k(B,C,D)$ be the function field over $k$ with three
variables, $U$ the affine diagonal quadric over $F$ defined by
$x^2+By^2+Cz^2+D=0$. For $P=(b,c,d) \in k^{\ast}\times k^{\ast} \times
k^{\ast}$, let $U_P$ be the affine quadric over $k$ defined by
$x^2+by^2+cz^2+d=0$. For $e$ in $\Br(U)$, we will define its
specialization $\sp(e;P)$ in $\Br(U_P)/\Br(k)$ in Section 3.2. Define
the following domain of specialization:
\[
 \cP_k :=\{P \in k^{\ast} \times k^{\ast} \times k^{\ast} \mid
 \Br(U_P)/\Br(k) \cong \bZ/2\bZ\}.
\]
First we will prove the following
\begin{thm}\label{thm1}
 $\Br(U)/\Br(F)=0$.
\end{thm}

As a corollary of this result, we have the following non-existence
result:
\begin{cor}\label{cor1}
 Assume that $k$ is not $2$-closed. Then there does not exist a pair of an
 element $e \in \Br(U)$ and a dense open subset $W \subset (\bG_{m,k})^3$ 
 satisfying the following conditions:
\begin{itemize}
 \item $\sp(e;\cdot)$ is defined on $W(k) \cap \cP_k$; 
 \item for all $P \in W(k) \cap \cP_k$, $\sp(e;P)$ is a generator of $\Br(U_P)/\Br(k)$. 
\end{itemize}
\end{cor}

We remark that the non-$2$-closedness of $k$ assures the Zariski
density of the image of $\cP_k$ in $(\bG_{m,k})^3$.  

We also note that there exists a method of producing an explicit symbolic generator
of affine quadrics {\it one by one}. For example, in the
paper~\cite{colliot-thelene09:_brauer_manin}, Section 5.8,
J.-L.~Colliot-Th\'el\`ene and F.~Xu gave such a construction under the
assumption that there exists a rational point on it. 
Our theorem asserts that such a construction cannot be done {\it simultaneously} for the
family of affine diagonal quadrics in an algebraic way.
 
This paper is organized as follows. In Section 2, we compute some
invariants of affine diagonal quadrics, in particular, their Picard
groups and their Galois structures. In the last of this section, we will
find an explicit generating cocycle of $H^1(k,\Pic(\bar{U}))$, which
plays an essential role in Section 3. In Subsection 3.1, we prove
Theorem \ref{thm1} by explicit calculations. Almost all arguments are
quite similar to arguments in \cite{uematsu14:_brauer}. In Subsection
3.2, we briefly recall the definition of specialization of Brauer groups,
introduced in \cite{uematsu14:_brauer}. Finally we prove Corollary
\ref{cor1}, the nonexistence of uniform generators.
 
{\bf Notation.} 
In this article, all fields are assumed to be of characteristic $0$. In
particular, they are infinite fields. 
For a field $k$, we fix its algebraic closure $\bar{k}$. we denote $\tilde{k}$ by
 $\bigcup_{n>0} k(\zeta_n)$, where $\zeta_n$ is a primitive $n$-th root
 of unity. If $k$ is a discrete valuation field, $k^{\ur}$ denotes the
 maximal unramified extension of $k$.
If $k$ contains $\zeta_n$, we denote $\{\cdot,\cdot\}_n$ by the
following composite of maps:
\[
 K^M_2(k)/n \to H^2(k,\mu_n^{\otimes 2}) \cong H^2(k,\mu_n) \inj \Br(k),
\] 
where the first one is a norm residue symbol of $k$, the second one is
induced by $\zeta_n^i \otimes \zeta_n^j \mapsto \zeta_n^{ij}$. We often
omit the subscripts $n$ unless confusing.

For a group $M$ and an endomorphism $f$ of $M$, the symbol $\lsub{f}{M}$
means the kernel of $f$. 

The Brauer group $\Br(X)$ of a scheme $X$ means the \'etale cohomology
group $H_{\et}^2(X,\bG_m)$. If $\pi\colon X \to \Spec k$ is a
$k$-scheme, we denote $\Gamma(X,\cO_X)$, $X \times_{k} \bar{k}$ and
$\Br(X)/\pi^{\ast}\Br(k)$ by $k[X]$, $\bar{X}$ and $\Br(X)/\Br(k)$
respectively.  

\section{Affine diagonal quadrics}\label{S-pdcs}
In this section, we are concerned with some invariants of affine
diagonal quadrics.
Let $U$ be the affine surface over $k$ defined by an equation
\begin{equation*}
 ax^2+by^2+cz^2+d=0,
\end{equation*}
where $a$, $b$, $c$ and $d$ are in $k^{\ast}$. Let $\pi\colon U\to \Spec k$
denote the structure morphism. Now we put
\begin{equation*}
 \lambda=-\dfrac{b}{a},\quad \mu=-\dfrac{c}{a}\quad \text{and} \quad \nu=\dfrac{ad}{bc},
\end{equation*}
and then we can write as the equation of $U$ 
\begin{equation*}
 x^2-\lambda y^2-\mu z^2+\lambda\mu\nu =0.
\end{equation*}

Put $\alpha=\sqrt{\lambda}$, $\gamma=\sqrt{\nu}$, $\alpha'=\sqrt{\mu}$
and $\beta=\alpha\gamma$. We also define $k'$ and $k''$ as
$k(\alpha,\gamma)$ and $k'(\alpha')$.

Let $X$ be a natural compactification of $U$, that is,
\[
 X = \Proj k[x,y,z,t]/(ax^2+by^2+cz^2+dt^2) \subset \P^3_k.
\]

Put $i\colon Z:=X\setminus U \inj X$. Define the following four lines on $\bar{X}$:
\begin{align*}
 L_1  &\colon x+\alpha y= z+\beta t=0, \\
 L_2  &\colon x+\alpha y= z-\beta t=0, \\
 L_1' &\colon x-\alpha y= z-\beta t=0, \\
 L_2' &\colon x-\alpha y= z+\beta t=0. 
\end{align*}
By abuse of notation, we write the restriction of $L_i$ and $L_i'$ on
$\bar{U}$ as the same symbols. 

First we compute the structure of $\bar{k}[U]^{\ast}, \Pic(\bar{U})$ and
$\Br(\bar{U})$. 
\begin{prop}\label{prop:H012} 
We have the following.\\
(1) $\bar{k}[\bar{U}]^{\ast}=k^{\ast}$.\\
(2) $\Pic(\bar{U})=\bZ[L_1]$.\\
(3) $\Br(\bar{U})=0$.
\end{prop}
\begin{proof}
We note that the pair $(\bar{X},\bar{Z})$ satisfies the purity of Brauer groups in
 the sense of \cite{grothendieck1968brauer3}, Section 6. Thus using the Leray spectral sequence
 $H^p(\bar{Z},R^qi^{!}\bG_m) \Rightarrow H^{p+q}_{\bar{Z}}(\bar{X},\bG_m)$ and the
 localization sequence, we have
 the following exact sequence (see \cite{grothendieck1968brauer3},
 Corollarie 6.2.):
\begin{align*}
 0 &\to \bar{k}^{\ast} \to \bar{k}[U]^{\ast} \to \bZ
   \to \Pic(\bar{X}) \to \Pic(\bar{U}) \to 0 
   \to \Br(\bar{X}) \to \Br(\bar{U}) \to H^1(\bar{Z}, \bQ/\bZ).
\end{align*}
In the above sequence the map $\bZ \to \Pic(\bar{X})$ maps $1$ to the
 class $[\bar{Z}]$. Since the defining equation of $\bar{X}$ can be written as
\[
 (x+\alpha y)(x-\alpha y)=(\alpha'z+\alpha'\beta
 t)(\alpha'z-\alpha'\beta t),
\]
we see that $\bar{X} \cong \P^1_{\bar{k}} \times
 \P^1_{\bar{k}}$ and 
\[
 \Pic(\bar{X}) = \bZ[L_1] \oplus \bZ[L_2].
\]
We can easily see that the class $[\bar{Z}]$ is equal to $[L_1]+[L_2]$
 in $\Pic(\bar{X})$. Therefore we have $\Pic(\bar{U})=\bZ[L_1]$ and the
 relation $[L_1]+[L_2]=0$ in $\Pic(\bar{U})$. This proves (2).

By this description, we have the injectivity of $\bZ \to \Pic(\bar{X})$,
 which implies $\bar{k}^{\ast} \cong \bar{k}[\bar{U}]^{\ast}$. This proves
 (1).

Finally, we see that $H^1(\bar{Z}, \bQ/\bZ)=0$ by the fact $\bar{Z}$ is
 isomorphic to $\P^1_{\bar{k}}$. This proves $\Br(\bar{X}) \cong
 \Br(\bar{U})$. Moreover, by the birational invariance of the Brauer
 group of proper varieties and the triviality of the Brauer group of
 projective spaces, we obtain $\Br(\bar{X})=0$. Therefore we complete the
 proof of (3) and Proposition \ref{prop:H012}.
\end{proof} 

Next we compute the Galois cohomology $H^1(k,\Pic(\bar{U}))$. 
\begin{prop}\label{prop:h1Pic}
\[
 H^1(k,\Pic(\bar{U})) \cong 
 \begin{cases}
  0 & \quad \text{if $\nu \in (k^{\ast})^2,$} \\
  \bZ/2\bZ & \quad \text{otherwise.}
 \end{cases}
\]
\end{prop}
\begin{proof}
Using the Hochschild-Serre spectral sequence
\[
 E^{p,q}_2=H^p(k',H^q(\bar{U},\bG_m)) \Rightarrow H^{p+q}(U_{k'},\bG_m)
\]
and the facts $U_{k'}(k') \neq \emptyset$ and $\bar{k}[\bar{U}]^{\ast}=\bar{k}^{\ast}$, we have 
\[
 \Pic(U_{k'}) \stackrel{\cong}{\to} \Pic(\bar{U})^{G_{k'}} \cong \bZ.
\]
Moreover, using the Hochschild-Serre spectral sequence
\[
 E^{p,q}_2=H^p(k'/k,H^q(k',\Pic(\bar{U}))) \Rightarrow H^{p+q}(k,\Pic(\bar{U}))
\]
and the triviality of $\Pic(\bar{U}) \cong \bZ$ as a $G_{k'}$-module, we have
\[
 H^1(k'/k,\Pic(U_{k'})) \cong H^1(k,\Pic(\bar{U})).
\]
Hence it suffices to compute the left hand side
 $H^1(k'/k,\Pic(U_{k'}))$. 

(i) In the case when $[k':k]=4$, put $s$ (resp. $t$) to be the generator of $\Gal(k'/k(\alpha))$
 (resp. $\Gal(k'/k(\gamma))$). Then we have $s\gamma=-\gamma$ and $t\alpha=-\alpha$.

Since $s\cdot [L_1] =[L_2]=-[L_1]$ in $\Pic U_{k'}$ we have
 $(\Pic(U_{k'}))^{k'/k(\alpha)}=0$. Using this fact, the spectral sequence
\[
 E^{p,q}_2=H^p(k(\alpha)/k,H^q(k'/k(\alpha),\Pic(U_{k'}))) \Rightarrow H^{p+q}(k'/k,\Pic(U_k'))
\]
induces the isomorphism $H^1(k'/k,\Pic(U_{k'})) \cong
 H^1(k'/k(\alpha),\Pic(U_{k'}))^{\Gal(k(\alpha)/k)}$. 

Now we compute $H^1(k'/k(\alpha),\Pic(U_{k'}))$. Using Tate cohomology,
 we have 
\[
 H^1(k'/k(\alpha),\Pic(U_{k'})) \cong
 \hat{H}^{-1}(k'/k(\alpha),\Pic(U_{k'})) \cong \lsub{N_s}{\Pic(U_{k'})}/I_s\Pic(U_{k'}),
\]
where for $\Gal(k'/k(\alpha))$-module $M$, $N_s\colon M \to M$ maps $m
 \to m+s\cdot m$ and $I_s\colon M \to M$ maps $m$ to $m-s\cdot m$. 

For $M=\Pic(U_{k'})$, we see that
\[
 \lsub{N_s}{\Pic(U_{k'})} = \Pic(U_{k'}), \quad I_s\Pic(U_{k'})
 =2\Pic(U_{k'}).
\]
Therefore we obtain
\[
 H^1(k'/k(\alpha),\Pic(U_{k'})) \cong \bZ/2\bZ\cdot[L_1].
\]
Since we have $\divi\left(\dfrac{z-\beta}{x+\alpha y}\right)=t\cdot L_1
 -L_1$ in $\Div(U_{k'})$, $[L_1] \in \Pic(U_{k'})$ is
 $t$-invariant. Thus we have
\[
 H^1(k'/k(\alpha),\Pic(U_{k'}))^{\Gal(k(\alpha)/k)} \cong \bZ/2\bZ.
\]

(ii) In the case when $k'=k(\gamma)$ and $\alpha \in k^{\ast}$, put $s$
 to be the generator of $\Gal(k(\gamma)/k)$. Then we have
 $s\cdot[L_1]=[L_2]=-[L_1]$ and therefore $H^1(k'/k,\Pic(U_{k'})) \cong
 \bZ/2\bZ\cdot [L_1]$. 

(iii) In the case when $k'=k(\alpha)=k(\gamma)$, put $s$ to be the
 generator of $\Gal(k(\gamma)/k)$. Then we have
 $s\cdot[L_1]=[L_2']=-[L_1]$ and therefore $H^1(k'/k,\Pic(U_{k'})) \cong
 \bZ/2\bZ\cdot[L_1]$. 

(iv) In the case when $k'=k(\alpha)$ and $\gamma \in k^{\ast}$, put $t$ to be the
 generator of $\Gal(k(\alpha)/k)$. Then we see that
\[
 H^1(k'/k, \Pic(U_{k'})) \cong
 \hat{H}^{-1}(k(\alpha)/k,\Pic(U_{k(\alpha)})) \cong 
 \lsub{N_t}{\Pic(U_{k(\alpha)})}/I_t\Pic(U_{k(\alpha)}) =0
\]
since $[L_1]$ is $t$-invariant. 

(v) In the case when $k'=k$, we have immediately
 $H^1(k'/k,\Pic(U_{k'}))=0$. 

Note that the condition $\nu \in (k^{\ast})^2$ is equivalent to (iv)
 $k'=k(\alpha)$ and $\gamma \in k^{\ast}$ or (v) $k'=k$, we complete the
 proof of Proposition \ref{prop:h1Pic}.
\end{proof}
By chasing the isomorphisms above, we can find an explicit cocycle of
$H^1(k'/k,\Pic(U_k'))$ corresponding to $-[L_1] \in \Pic(U_{k'})$:
\begin{cor}\label{cor:h1Pic}
Assume that $[k':k]=4$. If we define a 1-cocycle $\phi\colon \Gal(k'/k) \to \Pic(U_{k'})$ to be
\[
 \phi(t^i)=0, \quad \phi(st^i)=[L_1] \quad (i=0,1),
\]
then the class of $\phi$ in $H^1(k'/k,\Pic(U_{k'})) \cong \bZ/2\bZ$
 generates the whole group.
\end{cor}
\section{Non-existence of uniform generators}
\subsection{A vanishing theorem}
The main result of this paper is:
\begin{thm}\label{T-vb}
 Let $k$ be a field, $F=k(b,c,d)$ be the function field over $k$ with
 three variables $b,c,d$ and $U$ be the affine diagonal quadric
 $x^2+by^2+cz^2+d=0$ over $F$.
 Then
\begin{equation*} 
\Br(U)/\Br(F)=0.
\end{equation*} 
\end{thm}
\begin{proof}[Proof of Theorem \ref{T-vb}.]
By changes of coordinates, we may assume 
\[
 F=k(\lambda,\mu,\nu), \quad U\colon
 x^2-\lambda y^2-\mu z^2+\lambda\mu\nu=0
\]
We recall some notations. We define
\begin{equation*}
\alpha=\sqrt{\lambda}, \quad \gamma=\sqrt{\nu}, \quad
 \alpha'=\sqrt{\mu},\quad \beta=\alpha\gamma.
\end{equation*}
Moreover we put
\begin{equation*}
 F'=F(\alpha,\gamma),\quad F''=F'(\alpha)=F(\alpha,\gamma,\alpha').
\end{equation*}
We have the following exact sequence:
\begin{equation*}
 0 \to \Br(U)/\Br(F) \to H^1(F,\Pic(\bar{U})) \stackrel{d^{1.1}}{\to} H^3(F,\bar{F}^{\ast}).
\end{equation*}
Therefore, to prove the theorem, it suffices to show the image of a
 generator in $H^1(F,\Pic(\bar{U})) \cong \bZ/2\bZ$ does not vanish in $H^3(F,
 \bar{F}^{\ast})$. We divide its proof into 4 steps. \\
{\bf Step 1.} The goal of this step is to prove the following.
\begin{lem}\label{L:step1}
Let Let $\phi: \Gal(F'/F) \to \Pic(U_{F'})$ be the 1-cocycle
 appearing in Corollary \ref{cor:h1Pic}. Then its image under the map $d^{1,1}\colon
 H^1(F'/F,\Pic(U_{F'})) \to H^3(F'/F, (F')^{\ast})$ is
 represented by the 3-cocycle $\Phi$ defined to be the following equations:
\begin{alignat*}{3}
& \Phi(t^{j_1},t^{j_2},t^{j_3})=1 \\
& \Phi(1,s^{i_2}t^{j_2},s^{i_3}t^{j_3})=1,\quad 
&& \Phi(s^{i_1}t^{j_1},1,s^{i_3}t^{j_3})=1,\quad
&& \Phi(s^{i_1}t^{j_1},s^{i_2}t^{j_2},1)=1 \\
& \Phi(t,t,s)=\mu,
&& \Phi(t,s,s)=1,
&& \Phi(t,st,s)=\mu^{-1}, \\
& \Phi(s,t,s)=\mu,
&& \Phi(s,s,s)=1,
&& \Phi(s,st,s)=\mu^{-1}, \\
& \Phi(st,t,s)=1,
&& \Phi(st,s,s)=1,
&& \Phi(st,st,s)=1, 
\end{alignat*}
\[
 \Phi(s^{i_1}t^{j_1},s^{i_2}t^{j_2},s^{i_3}t^{j_3})
=\Phi(s^{i_1}t^{j_1},s^{i_2}t^{j_2},s^{i_3}),
\]
where the indices $i_{\ast}$ and $j_{\ast}$ take on any values in $\{0, 1\}$.
\end{lem}
\begin{proof}
Define $\cD$ to be the $\Gal(F'/F)$-orbit of $\bZ L_1$. In other words, 
\[
 \cD = \bZ L_1 \oplus \bZ L_2 \oplus \bZ L_1' \oplus \bZ L_2'.
\]
We also define $\cD_0$ to be the $\Gal(F'/F)$-submodule of $\cD$
 generated by the following three principal divisors:
\begin{align*}
 D_1 =\divi(f_1) =L_1+L_2, \quad &f_1:=x+\alpha y,\\
 D_2 =\divi(f_2) =L_1'+L_2', \quad &f_2:=x-\alpha y,\\
 D_3 =\divi(f_3) =L_1+L_2', \quad &f_3:=z+\beta.
\end{align*}

Then we can see that the sequences 
\begin{align*}
  &0 \to \cD_0 \to \cD \to \Pic(U_{F'}) \to 0,\\
  &1 \to (F')^{\ast} \to \divi^{-1}(\cD_0) \to \cD_0 \to 0
\end{align*}
are both exact. These induce the following connecting homomorphisms 
\begin{align*}
&\partial\colon H^1(F'/F,\Pic(V_{F'})) \to H^2(F'/F,\cD_0), \\
&\delta\colon H^2(F'/F,\cD_0) \to H^3(F'/F, F'^{\ast}).
\end{align*}
By the same argument as in \cite{kresch2008effectivity}, Proposition 6.1, (i), we have the
 equation $\delta\circ\partial=d^{1,1}$. Therefore the explicit
 computation of cocycles completes the proof of Lemma \ref{L:step1}.  
\end{proof}
\noindent
{\bf Step 2.} The goal of this step is the following: 
\begin{lem}\label{h3mu2}
 The image of $\Phi$ under the inflation $i^{F'}_{\bar{F}}\colon
 H^3(F'/F,(F')^{\ast}) \to H^3(F,\bar{F}^{\ast})$ comes from
 $H^3(F''/F,\mu_2).$ 
\end{lem}
\begin{proof}
 We have the exact sequence of $\Gal(F''/F)$-modules
\begin{equation*}
 1 \to \mu_2 \to F''^{\ast} \stackrel{2}{\to} (F''^{\ast})^2 \to 1,
\end{equation*}
and we get the following commutative diagram
\begin{equation*}
\xymatrix{
 & H^3(F'/F,F'^{\ast}) \ar^{i^{F'}_{F''}}[d] &  \\
 H^3(F''/F,\mu_2) \ar[d] \ar[r] & H^3(F''/F,F''^{\ast})
  \ar^{i^{F''}_{\bar{F}}}[d] \ar^{2}[r] & H^3(F''/F,(F''^{\ast})^2) \ar[d]\\
 H^3(F,\mu_2) \ar[r] & H^3(F,\bar{F}^{\ast}) \ar^{2}[r] &
  H^3(F,\bar{F}^{\ast}), \\
}
\end{equation*}
where $i_{F'}^{F''}$ and $i^{F''}_{\bar{F}}$ are inflations and each row
 is exact. Therefore to prove the claim, it suffices to
 show $i^{F'}_{F''}[\Phi]$ vanishes in
 $H^3(F''/F,(F''^{\ast})^2)$. Let $w$ be the generator of $\Gal(F''/F')$. The image of
 $i^{F'}_{F''}[\Phi]$ under $2\colon H^3(F''/F,F''^{\ast})
 \to H^3(F''/F,(F''^{\ast})^2)$ is the class of the following cocycle:
\begin{equation*}
 (s^{i_1}t^{j_1}w^{k_1},s^{i_2}t^{j_2}w^{k_2},s^{i_3}t^{j_3}w^{k_3})
  \mapsto \Phi(s^{i_1}t^{j_1},s^{i_2}t^{j_2},s^{i_3}t^{j_3})^2,
\end{equation*}
and what we have to prove is that this cocycle is in $B^3(F''/F,(F''^{\ast})^2)$.
Define $\psi \in C^2(F''/F,(F''^{\ast})^2)$ to be:
\begin{alignat*}{2}
&\psi(t^{j_1}w^{k_1},t^{j_2}w^{k_2})=1, &\quad 
&\psi(w^{k_1},st^{j_2}w^{k_2})=1, \\
&\psi(tw^{k_1},st^{j_2}w^{k_2})=\mu, &\quad 
&\psi(st^{j_1}w^{k_1},t^{j_2}w^{k_2})=1, \\
&\psi(sw^{k_1},st^{j_2}w^{k_2})=\mu, &\quad 
&\psi(stw^{k_1},st^{j_2}w^{k_2})=1, 
\end{alignat*}
where indices $i_{\ast}$, $j_{\ast}$ and $k_{\ast}$ take on any values in $\{0, 1\}$.
Then we can easily see $d\psi = (i^{F'}_{F''}\Phi)^2$ in
 $C^3(F''/F,(F''^{\ast})^2)$ and hence the class of $i^{F'}_{F''}\Phi$
 vanishes in $H^3(F''/F,(F''^{\ast})^2)$. This completes the proof of
 Proposition \ref{h3mu2}. 
\end{proof}
By using this cochain $\psi$, we can construct the class in
 $H^3(F''/F,\mu_2)$ whose image in $H^3(F,\bar{F}^{\ast})$ is
 $i^{F'}_{\bar{F}}\Phi$. We have the following diagram with
 exact rows
\begin{equation*}
\xymatrix{
0 \ar[r] & 
C^2(F''/F,\mu_2) \ar_{d^2}[d]\ar^{f^2}[r] & 
C^2(F''/F,F''^{\ast}) \ar_{d^2}[d]\ar^{g^2}[r] & 
C^2(F''/F,(F''^{\ast})^2) \ar_{d^2}[d]\ar[r] &
0 \\
0 \ar[r] & 
C^3(F''/F,\mu_2) \ar_{d^3}[d]\ar^{f^3}[r] & 
C^3(F''/F,F''^{\ast}) \ar_{d^3}[d]\ar^{g^3}[r] & 
C^3(F''/F,(F''^{\ast})^2) \ar_{d^3}[d]\ar[r] &
0 \\
0 \ar[r] & 
C^4(F''/F,\mu_2) \ar^{f^4}[r] & 
C^4(F''/F,F''^{\ast}) \ar^{g^4}[r] & 
C^4(F''/F,(F''^{\ast})^2) \ar[r] &
0. \\
}
\end{equation*}
If $\tilde{\psi} \in C^2(F''/F,F''^{\ast})$ is a lift of $\psi$, we see
 that there exists a cochain $\Phi' \in C^3(F''/F,\mu_2)$ such that
\begin{equation*}
 f^3\Phi'=i^{F'}_{F''}\Phi - d^2\tilde{\psi}.
\end{equation*}
Moreover, by construction, $\Phi'$ is a cocycle and
 $f^3[\Phi']=[i^{F'}_{F''}\Phi-d^2\tilde{\psi}]=i^{F'}_{F''}[\Phi]$. Therefore
 the class $[\Phi']$ is what we need.  As a lift $\tilde{\psi}$ of
 $\psi$, we can take the following natural cochain:
\begin{alignat*}{2}
&\tilde{\psi}(t^{j_1}w^{k_1},t^{j_2}w^{k_2})=1, &\quad 
&\tilde{\psi}(w^{k_1},st^{j_2}w^{k_2})=1, \\
&\tilde{\psi}(tw^{k_1},st^{j_2}w^{k_2})=\alpha', &\quad 
&\tilde{\psi}(st^{j_1}w^{k_1},t^{j_2}w^{k_2})=1, \\
&\tilde{\psi}(sw^{k_1},st^{j_2}w^{k_2})=\alpha', &\quad 
&\tilde{\psi}(stw^{k_1},st^{j_2}w^{k_2})=1.
\end{alignat*}
Hence we can write $\Phi'$ explicitly as follows:
\begin{equation*}
 (s^{i_1}t^{j_1}w^{k_1},s^{i_2}t^{j_2}w^{k_2},s^{i_3}t^{j_3}w^{k_3})
  \mapsto
  \dfrac{\tilde{\psi}(s^{i_2}t^{j_2}w^{k_2},s^{i_3}t^{j_3}w^{k_3})}{w^{k_1}\tilde{\psi}(s^{i_2}t^{j_2}w^{k_2},s^{i_3}t^{j_3}w^{k_3})}
  \in \mu_2.
\end{equation*} 
\noindent
{\bf Step 3.} 
In this step, we reduce the proof of the nontriviality of
 $i_{\bar{F}}^{F'}[\Phi]$ in $H^3(F,\bar{F}^{\ast})$ to that of the
 nontriviality of a class in some $H^2$ cohomology. For any prime divisor $D \subset \bA_k^3=\Spec
 k[\lambda,\mu,\nu]$, we have the following commutative diagram:
\begin{equation*}
\xymatrix{
H^3(F''/F,\mu_2) \ar^{i^{F''}_{\bar{F}}}[d]& \\
H^3(F,\mu_2) \ar[d] \ar^(0.45){\res_D}[r] & H^2(k(D),\bZ/2\bZ) \ar[d]\\
H^3(F,\bQ/\bZ(1)) \ar_{\cong}[d] \ar^(0.45){\res_D}[r] & H^2(k(D),\bQ/\bZ) \\
H^3(F,\bar{F}^{\ast}), &
}
\end{equation*}
where $F=k(\lambda,\mu,\nu)$ is considered as the function field of
 $\bA_k^3$, $k(D)$ is the function field of $D$, and $\res_D$ are residue
 maps associated to $D$.

Thus to prove the theorem, it suffices to show:
\begin{equation}\label{E-rH2}
\text{There exists } D \subset \bA^3_{k} \text{ such that }
 \res_D(i^{F''}_{\bar{F}}[\Phi']) \neq 0 \in H^2(k(D),\bQ/\bZ).
\end{equation} 

In the sequel, $D$ always denotes the divisor $\{\mu=0\} \subset
 \bA^3_k$. Let $\cO_D$ be the completion of the local ring
 $k[\lambda,\mu,\nu]_{(\mu)}$ at its maximal ideal and $F_D$ its
 fractional field. Note that $\mu$ is a uniformizer of $\cO_D$ and the
 residue field of $\cO_D$ is isomorphic to $k(D)=k(\lambda,\nu)$.
 
 Now we should recall the definition of $\res_D$:
 \begin{lem}[\cite{garibaldi2003cohomological}, III, Theorem 6.1.]\label{L-dr}
Let $K$ be a complete discrete valuation field of characteristic $0,$
  $\kappa$ its residue field$,$ $p$ its characteristic$,$ and $I \subset
  G_K$ its inertia$.$ Then for any $G_{\kappa}$-module $C$ without having
  non-zero $p$-torsion element$,$ we have the following exact sequence
\begin{equation*}
0 \to H^i(\kappa,C) \to H^i(K,C) \stackrel{r}{\to}
 H^{i-1}(\kappa,\Hom(I,C)) \to 0.
\end{equation*}
Here the second map is induced by the canonical map $G_K \to G_{\kappa},$
and $r$ is defined as follows$:$
\begin{quote}
For a normalized cocycle $\phi \in Z^i(K,C)$ satisfying
\begin{equation*}
\text{for all } i \geq 2,\quad g_i \equiv g_i' \mod I \Rightarrow
 \phi(g_1,g_2,\ldots,g_n)=\phi(g_1,g_2',\ldots, g_n'),
\end{equation*}
define $r\phi \in Z^{i-1}(\kappa,\Hom(I,C))$ as$:$
\begin{equation*}
\text{for all } h \in I,\quad
 (r\phi)(\bar{g_1},\ldots,\bar{g_{n-1}})(h)=\phi(h,g_1,\ldots,g_{n-1}),
\end{equation*}
where $g_i$ are lifts of $\bar{g_i}$ to $G_K.$
\end{quote}
\end{lem}
There is the canonical isomorphism
\begin{equation*}
\iota\colon \Hom(G_{F_D^{\ur}},\mu_2) =H^1(F_D^{\ur},\mu_2) \cong
 {F_D^{\ur}}^{\ast}/({F_D^{\ur}}^{\ast})^2 \cong \bZ/2\bZ,
\end{equation*}
where the middle isomorphism is induced by Kummer sequence and the right one
 is given by normalized valuation on $F_D^{\ur}$. Then $\res_D$ is given by
\begin{equation*}
H^3(F,\mu_2)\to H^3(F_D,\mu_2) \stackrel{r}{\to}
 H^2(k(D),\Hom(G_{F_D^{\ur}},\mu_2)) \stackrel{H^2(\iota)}{\to} H^2(k(D),\bZ/2\bZ).
\end{equation*}  
Now we describe the class $r[i^{F''}_{\bar{F}}\Phi'] \in
 H^2(k(D),\Hom(G_{F^{\ur}_D},\mu_2))$ explicitly. We introduce some
 field extensions. Let $k(D)'$, $F_D''$, $F_D'$ be the same notation as
 in \S \ref{S-pdcs}. Moreover, by abuse of notation, we denote the
 elements in $\Gal(F_D''/F_D)$ corresponding to $s$, $t$ and $w \in
 \Gal(F''/F)$ as the same symbols.  To make our situation clear, we give
 the following diagram of field extensions:
\begin{equation*}
\xymatrix{
& & F_D'' \ar@{.>}^(0.45){{\rm residue\ field}}[rr]
\ar@{-}^{\rm ramified}_{2}[d] & & k(D)' \ar@{=}[d] \\
F'' \ar[urr] \ar@{-}_{2}[d] & & F_D' \ar@{.>}[rr]
 \ar@{-}^{\rm unramified}_{4}[d] & &
 k(D)' \ar@{-}_{4}[d] \\
F' \ar[urr] \ar@{-}_{4}[d] & & F_D \ar@{.>}[rr] & & k(D) \\
F \ar_(0.4){{\rm completion}}[urr] & & & &\\
}
\end{equation*}

We have the following:
\begin{lem}\label{L-rrphi}
If we define the cochain
\begin{equation*}
\bar{r\Phi'} \in C^2(k(D)'/k(D),\Hom(\Gal(F_D''/F_D'),\mu_2))
\end{equation*}
as
\begin{equation*}
\bar{r\Phi'}(\bar{s}^{i_1}\bar{t}^{j_1},\bar{s}^{i_2}\bar{t}^{j_2})(w^k):=\Phi'(w^k,s^{i_1}t^{j_1},s^{i_2}t^{j_2}),
\end{equation*}
where $\bar{s}$ and $\bar{t}$ is the image of $s$ and $t$ under the
 natural map
\begin{equation*}
  \Gal(F_D''/F_D) \to \Gal(k(D)'/k(D)),
\end{equation*}
then $\bar{r\Phi'}$ is a cocycle and its image under the map
\begin{equation*}
  i^{k(D)'}_{\bar{k(D)}}\colon H^2(k(D)'/k(D),\Hom(\Gal(F_D''/F_D'),\mu_2)) \to
   H^2(k(D),\Hom(G_{F_D^{\ur}},\mu_2))
\end{equation*}
is $ri^{F''}_{\bar{F}}[\Phi'].$
\end{lem}
By using the following isomorphisms of $\Gal(k(D)'/k(D))$ and
 $G_{k(D)}$-modules
\begin{equation*}
\xymatrix{
\bZ/2\bZ \ar@{=}[d] \ar^(0.3){\cong}[r] & \Hom(\Gal(F_D''/F_D'),\mu_2) \ar[d]
\ar[r] &
\Hom(G_{F_D^{\ur}},\mu_2) \ar[d] \ar^(0.55){\cong}[r] & \bZ/2\bZ \ar[d] \\
\bZ/2\bZ \ar^(0.25){\cong}[r] & \Hom(\Gal(F_D''/F_D'),\bQ/\bZ(1)) \ar[r] &
\Hom(G_{F_D^{\ur}},\bQ/\bZ(1)) \ar^(0.65){\cong}[r] & \bQ/\bZ,
}
\end{equation*}
we have the following diagram
\begin{equation*}
\xymatrix{
H^2(k(D)'/k(D),\bZ/2\bZ) \ar@{=}[d] \ar[r] &
H^2(k(D),\bZ/2\bZ) \ar[d]\\
H^2(k(D)'/k(D),\bZ/2\bZ) \ar[r] &
H^2(k(D),\bQ/\bZ).
}
\end{equation*}
Put $E=\tilde{k(D)}$ and $E'=\tilde{k(D)'}$. Noting that $k(D)'=k(D)(\alpha,\gamma)$ and that $\alpha$ and
 $\gamma$ are transcendental over $k$, we have $k(D)' \cap E = k(D)$ and therefore
\begin{equation*}
 \Gal(E'/E) \stackrel{\cong}{\to} \Gal(k(D)'/k(D)).
\end{equation*}
We fix an isomorphism $\bQ/\bZ \cong \bQ/\bZ(1)$ as trivial
 $G_E$-modules. Then we have the following commutative diagram:
\begin{equation*}
\xymatrix{
H^2(k(D)'/k(D),\bZ/2\bZ) \ar_{\cong}[d] \ar[r] & H^2(k(D),\bQ/\bZ) \ar[d] \\
H^2(E'/E,\bZ/2\bZ) \ar_{\cong}[d] \ar[r] &
 H^2(E,\bQ/\bZ) \ar^{\cong}[d] \\
H^2(E'/E,\mu_2) \ar[d] \ar[r] &
H^2(E,\bQ/\bZ(1)) \ar^{\cong}[d] \\
H^2(E,\mu_2) \ar[r] & H^2(E,E^{\ast}).
}
\end{equation*}
Since the bottom map in the above diagram is injective by Hilbert's
 Theorem 90, in order to prove the claim \eqref{E-rH2}, it suffices to
 show that $[\bar{r\Phi'}] \in H^2(k(D)'/k(D),\bZ/2\bZ)$ is non-trivial in
 $H^2(E,\mu_2)$. We define the cocycle $\Psi \in
 Z^2(E'/E,\mu_2)$ as follows:
\begin{alignat*}{2}
&\Psi(t^{j_1}, t^{j_2})=1, &&\quad 
 \Psi(1,st^{j_2})=1, \\
&\Psi(t, st^{j_2})=-1, &&\quad 
 \Psi(st^{j_1},st^{j_2})=1, \\
&\Psi(s, st^{j_2})=-1, &&\quad 
 \Psi(st,st^{j_2})=1.
\end{alignat*}
We can easily see that $[\Psi]$ is the image of $[\bar{r\Phi'}] \in H^2(k(D)'/k(D),\bZ/2\bZ)$
under the isomorphism $H^2(k(D)'/k(D),\bZ/2\bZ) \stackrel{\cong}{\to}
 H^2(E'/E,\mu_2)$ in the above diagram. Hence we have to
 prove the nontriviality of $i^{E'}_{\bar{E}}[\Psi] \in
 H^2(E,\mu_2)$, where
\begin{equation*}
i^{E'}_{\bar{E}}\colon H^2(E'/E,\mu_2) \to H^2(E,\mu_2).
\end{equation*}
\noindent
{\bf Step 4.} In this step, we reduce the proof of the nontriviality of
 $i^{E'}_{\bar{E}}[\Psi]$ in $H^2(E,\mu_2)$ to that of the
 nontriviality of a class in some $H^1$ cohomology, and finish the proof
 of Theorem \ref{T-vb}. 
We consider its residue along to the divisor $D'=\{\lambda=0\} \subset \bA^2_{k}$. 
We fix notations. Let $\cO_{D'}$ be the completion of the local ring
$k[\lambda,\nu]_{(\lambda)}$ at its maximal ideal and $E_{D'}$ its fractional
field. Note that $\lambda$ is a uniformizer of $\cO_{D'}$ and the residue
field of $\cO_{D'}$ is isomorphic to $k(D')=k(\mu)$. Let $E_{D'}'$ be the same notation as in \S
\ref{S-pdcs}. By abuse of notation, we denote the elements in
$\Gal(E_{D'}'/E_{D'})$ corresponding to $s$ and $t \in \Gal(E'/E)$ as the
same symbols.  To make our situation clear, we give the following
diagram of field extensions:
\begin{equation*}
\xymatrix{
& & E_{D'}' \ar@{.>}^(0.45){{\rm residue\ field}}[rr]
\ar@{-}^{\rm ramified}_{2}[d] & & k(D')(\gamma) \ar@{=}[d] \\
E' \ar[urr] \ar@{-}_{2}[d] & & E_{D'}(\gamma) \ar@{.>}[rr]
 \ar@{-}^{\rm unramified}_{2}[d] & &
 k(D')(\gamma) \ar@{-}_{2}[d] \\
E(\gamma) \ar[urr] \ar@{-}_{2}[d] & & E_{D'} \ar@{.>}[rr] & & k(D'). \\
E \ar_(0.4){{\rm completion}}[urr] & & & &\\
} 
\end{equation*}
Now $\res_{D'}$ is given by
\begin{equation*}
H^2(E,\mu_2)\to H^2(E_{D'},\mu_2) \stackrel{r}{\to}
 H^1(k(D'),\Hom(G_{E_{D'}^{\ur}},\mu_2)) \stackrel{\cong}{\to} H^1(k(D'),\bZ/2\bZ).
\end{equation*}
We have a similar result to Lemma \ref{L-rrphi}:
\begin{lem}\label{L-cc}
If we define the cochain
\begin{equation*}
\bar{r\Psi} \in C^1(k(D')(\gamma)/k(D'),\Hom(\Gal(E_{D'}'/E_{D'}(\gamma)),\mu_2))
\end{equation*}
as
\begin{equation*}
\bar{r\Psi}(\bar{s}^{i})(t^j):=\Psi(t^j,s^i),
\end{equation*}
where $\bar{s}$ is the image of $s$ under the
 natural map $\Gal(E_{D'}'/E_{D'}) \to \Gal(k(D')(\gamma)/k(D'))$,
then $\bar{r\Psi}$ is a cocycle and its image under the map
\begin{equation*}
  i^{k(D')(\gamma)}_{\bar{k(D')}}\colon H^1(k(D')(\gamma)/k(D'),\Hom(\Gal(E_{D'}'/E_{D'}(\gamma)),\mu_2)) \to
   H^1(k(D'),\Hom(G_{E_{D'}^{\ur}},\mu_2))
\end{equation*}
is $ri^{E'}_{\bar{E}}[\Psi].$
\end{lem}
By the injectivity of $i_{\bar{k(D')}}^{k(D')(\gamma)}$ and $[\bar{r\Psi}] \neq
0$, we have the nontriviality of the image of $[\bar{r\Psi}]$ in
 $H^1(k(D'),\bZ/2\bZ)$. 
Therefore we complete the proof of Theorem \ref{T-vb}.
\end{proof}

\subsection{Specialization of Brauer groups}\label{S-sbg}
In \cite{uematsu14:_brauer}, we introduce the notion of uniform
generators. For example, the Brauer group of projective diagonal cubic surfaces of the form
$x^3+y^3+z^3+dt^3=0$ has uniform symbolic generators, that is, if we put 
\begin{equation*}
e_1(D)=\left\{D,\dfrac{x+\zeta y}{x+y}\right\}_3,\quad 
e_2(D)=\left\{D,\dfrac{x+z}{x+y}\right\}_3 
\end{equation*}
where $D$ is considered as an {\it indeterminate} and if we want
 symbolic generators of $\Br(X_{d})/\Br(k)$, where $X_{d}$ is the
 surface of the form $x^3+y^3+z^3+dt^3=0$, we can
 get them by specializing $e_1(D)$ and $e_2(D)$ at $D=d$.

However, the Brauer group of ``general'' diagonal cubic surfaces
$ax^3+by^3+cz^3+dt^3=0$ does not have such an uniform generator. 

Our main result of this article is that a similar non-existence result holds for
affine diagonal quadrics. 

To formulate this uniformity and claim the precise statement, we briefly recall the definition of the
specialization of the Brauer group. For detail, see
\cite{uematsu14:_brauer}, Section 2. Let $k$ be a field, $\cO_F$ a polynomial ring over $k$ with $r$
variables, $F$ its fractional field and $f_1, \ldots f_m$ polynomials in
$\cO_F[x_1,\ldots,x_n]$.  Let $\cX$ be the scheme over $\cO_F$ defined as:
\begin{equation*}
\cX = \Spec\left(\cO_F[x_1,\ldots,x_n]/(f_1,\ldots,f_m)\right)
 \stackrel{\pi}{\to} \Spec \cO_F=\bA_k^r.
\end{equation*}    
Let $\pi_F\colon X:=\cX_F \to \Spec F$ be the base change of $\pi$ to
$\Spec F$. Assume that $X$ is smooth over $F$.  Let $e \in \Br(X)$ be an
arbitrary element. Then there exists a non-empty affine open subscheme
$S$ of $\bA_k^r$ and $\tilde{e}
\in \Br(\cX\times_{\bA_k^r}S)$ satisfying that $\cX \times_{\bA_k^r}S$ is
smooth over $S$ and that
\begin{equation*}
 \res^S_{\Spec F}(\tilde{e}) =e,
\end{equation*}
where $\res^S_{\Spec F}\colon \Br(\cX\times_{\bA^r_k}S) \to \Br(X).$ 
For a given $P \in S(k)$, we have the following diagram:
\begin{equation*}
\xymatrix{
X_P \ar^{P}[r] \ar^{\pi_0}[d] \ar@{}|{\square}[dr] & \cX \times_{\bA^r_k}S
 \ar^{\pi_S}[d] \ar@{}|{\square}[dr] & X \ar[l] \ar^{\pi_F}[d] \\
\Spec k \ar^{P}[r] & S & \Spec F \ar[l], \\
}
\end{equation*}
where $X_P:=\cX\times_{\bA^r_k}\Spec k$.
Now we can define {\it the specialization of $e$ at $P$} as
\begin{equation*}
\sp(e;P):=P^{\ast}\tilde{e} \in \Br(X_P).
\end{equation*}
Note that this definition is independent of $S$. 

\subsection{Proof of the non-existence result}
Finally we state the non-existence of uniform generators of the Brauer group of
affine diagonal quadrics. Let $k, F, U$ be as in Theorem \ref{T-vb}. To eliminate some exceptional points, we
introduce the following domain of specialization:
\[
 \cP_k:=\{P \in (\bG_{m,k})^3(k) \mid \Br(U_P)/\Br(k) \cong \bZ/2\bZ.\}
\]
We would like to assume that the image of $\cP_k$ in $(\bG_{m,k})^3$ is
Zariski dense in $(\bG_{m,k})^3$. This assumption is characterized by the
non-$2$-closedness of $k$: 
\begin{prop}\label{not-2-closed}
Let $k$ be a field. Then the following
 conditions are equivalent$:$

{\rm(1)} $\cP_k$ is Zariski dense in $(\bG_{m,k})^3;$

{\rm(2)} $\cP_k$ is non-empty$;$

{\rm(3)} $k$ is not $2$-closed.
\end{prop}

\begin{proof}
(1) $\Rightarrow$ (2). This is a trivial implication.\\
(2) $\Rightarrow$ (3). We prove the contrapositive statement. If we assume
 that $k$ is $2$-closed, then the equation of affine diagonal quadrics is essentially equal to 
\[
 U\colon x^2+y^2+z^2+1=0.
\]
By Proposition \ref{prop:h1Pic}, we have $H^1(k,\Pic(\bar{U}))=0$ and therefore
 $\Br(U)/\Br(k)=0$. Hence we have $\cP_k=\emptyset$.\\
(3) $\Rightarrow$ (1). Since $\dim_{\bF_2}k^{\ast}/(k^{\ast})^2\geq 1$,
 we can take a non-trivial element $v \in k^{\ast}/(k^{\ast})^2$. Now we put $\cP$ as
\begin{equation*}    
 \cP=S\times S \times S, \quad S=v(k^{\ast})^2.
\end{equation*}
Now we can easily show that for any $P \in \cP$, $U_P(k)\neq \emptyset$ and
 $H^1(k,\Pic(\bar{U}))\cong \bZ/2\bZ$, which imply that $\cP \subset
 \cP_k$. Moreover, using the infiniteness of $k$, we also check that $\cP$ is Zariski dense in $(\bG_{m,k})^3$
 by Lemma 5.9 in \cite{uematsu14:_brauer}. Thus $\cP_k$ is Zariski dense
 in $(\bG_{m,k})^3$.  
\end{proof}
Using the definition of specialization, Proposition \ref{not-2-closed}
and Theorem \ref{T-vb}, we have the following:
\begin{cor}[Corollary of Theorem \ref{T-vb}]\label{C-main}
Let $k$ be a field which is not $2$-closed, $F=k(b,c,d)$ and $U$ be the
 affine diagonal quadrics over $F$
 defined by $x^2+b y^2+c z^2+d=0.$ Then there does not exist a pair of
 an element $e \in \Br(U)$ and a dense open subset $W \subset
 (\bG_{m,k})^3$ satisfying the following conditions$:$
\begin{itemize}
 \item $\sp(e;\cdot)$ is defined on $W(k) \cap \cP_k$;
 \item for all $P \in W(k) \cap \cP_k, \sp(e;P)$ is a generator of $\Br(U_P)/\Br(k).$
\end{itemize}
\end{cor}  
\begin{proof}
We would have an element $e$ and $W$
 satisfying the conditions stated in the above. By Theorem \ref{T-vb}, we have
\begin{equation*}
 \Br(U)/\Br(F)=0
\end{equation*}
and hence there exists an element $e' \in \Br(F)$ such that
 $\pi_F^{\ast}e'=e$. We have the isomorphism
 \begin{equation*}
  \injlim_{i} \Br(S_i) = \Br(F),
\end{equation*}
where $(S_i)$ is the projective system of the non-empty open affine
 subschemes in $\bA^3_{k}$. Put
 $\cU=\Spec\cO_F[x,y,z]/(x^2+by^2+cz^2+d))$. Then there exist a non-empty affine open
 subscheme $S \subset \bA_k^3$ and $\tilde{e'} \in \Br(S)$ such that $\tilde{e'}$ is a
 lift of $e'$ and $\cU\times_{\bA_k^3}S$ is smooth over $S$. Since $S$
 and $W$ are not empty, $S \cap W$ is also a non-empty Zariski open set
 in $(\bG_{m,k})^3$. Moreover, $\cP_k$ is a Zariski dense set in
 $(\bG_{m,k})^3$ by the non-$2$-closedness of $k$. Thus there
 exists a point $P \in (S \cap W)(k) \cap \cP_k$. Now we
 have the following commutative diagram:
\begin{equation*}
\xymatrix{
\Br(U_P) &
\ar^{P^{\ast}}[l] \Br(\cU\times_{\bA^3_k}S) \ar@{^(->}[r] & 
\Br(U) \\
\Br(k) \ar_{\pi P^{\ast}}[u] &
\ar^{P^{\ast}}[l] \Br(S) \ar_{\pi_{S}^{\ast}}[u] \ar@{^(->}[r]&
\Br(F) \ar_{\pi_{F}^{\ast}}[u], \\
}
\end{equation*}
and hence we can take $\pi_S^{\ast}\tilde{e'}$ as a lift of $e$. Then we
 get
\begin{equation*}
\sp(e;P) = P^{\ast}(\pi_S^{\ast}\tilde{e'}) =
 \pi_P^{\ast}P^{\ast}\tilde{e'} \in \pi_P^{\ast}\Br(k).
\end{equation*}
This means that $\sp(e;P)$ is zero in the group $\Br(U_P)/\Br(k)$,
 which contradicts that $\sp(e;P)$ is a generator of $\Br(U_P)/\Br(k)
 \cong \bZ/2\bZ$. 
\end{proof}
\begin{rem}
As is the case of projective diagonal cubic surfaces, we can prove that the Brauer group
 of affine diagonal quadrics of the form $x^2-y^2-cz^2+d=0$ with $cd
 \notin (k^{\ast})^2$ is
 isomorphic to $\bZ/2\bZ$, and that its uniform generator can be taken as
\[
 \left\{cd, x+y\right\}.
\] 
\end{rem}
\begin{ack} 
The author is grateful to Professor Timothy D.~Browning for suggesting the
 subject of this article.
\end{ack}

\bibliographystyle{amsalpha}
\bibliography{mathematics}
{\scriptsize
\textsc{Tetsuya Uematsu}\\
\textsc{Department of General Education, National Institute of
Technology, Toyota College}\\
\textsc{2-1 Eisei-cho Toyota Aichi 471-8525, JAPAN}\\
\textit{e-mail address}: utetsuya@08.alumni.u-tokyo.ac.jp}
\end{document}